\documentclass[a4paper, 11pt]{amsart}
\usepackage{amsmath,amsthm,amssymb}
\usepackage{graphicx,color}
\usepackage[all]{xy}
\usepackage{hyperref}
\usepackage[margin=2cm]{geometry}
\usepackage{tikz,enumerate}

\theoremstyle{definition}
\newtheorem{thm}{Theorem}[section]
\newtheorem{lem}[thm]{Lemma}
\newtheorem{cor}[thm]{Corollary}
\newtheorem{prop}[thm]{Proposition}

\newtheorem{dfn}[thm]{Definition}
\newtheorem{rem}[thm]{Remark}
\newtheorem{ex}[thm]{Example}

\newcommand{\R}{\mathbb{R}}
\newcommand{\Q}{\mathbb{Q}}

\newcommand{\Z}{\mathbb{Z}}
\newcommand{\C}{\mathbb{C}}

\newcommand{\e}{\mathbf{e}}

\newcommand{\row}{\mathrm{Row}}
\newcommand{\Sym}{\mathfrak{S}}

\newcommand{\Aut}{Aut}
\newcommand{\cB}{\mathcal{B}}
\newcommand{\fB}{\mathfrak{B}}

\newcommand{\RZ}{\mathbb{R}\mathcal{Z}}
\renewcommand{\Im}{\mathrm{Im}}

\newlength{\cellsize} 
\newsavebox{\cell}

\sbox{\cell}{%
\begin{picture}(18,18)\linethickness{0.6pt} %
  \put(0,0){\line(1,0){18}} \put(0,0){\line(0,1){18}}
  \put(18,0){\line(0,1){18}} \put(0,18){\line(1,0){18}}
\end{picture}}

\newcommand\cellify[1]{%
  \def\thearg{#1}\def\nothing{}%
  \ifx\thearg\nothing \vrule width0pt height\cellsize depth0pt\else
  \hbox to 0pt{\usebox{\cell} \hss}\fi%
  \vbox to \cellsize{ \vss \hbox to
  \cellsize{\hss$#1$\hss} \vss}
}
\newcommand\tableau[1]{\vtop{\let\\\cr
\baselineskip -16000pt \lineskiplimit 16000pt \lineskip 0pt
\ialign{&\cellify{##}\cr#1\crcr}}}

\newcommand\bas[1]{\omit \vbox to \cellsize{ \vss \hbox to \cellsize{\hss$#1$\hss} \vss}}

\title[Finite group representation on real toric spaces]{Geometric representations of finite groups on real toric spaces}

\author[S.~Cho]{Soojin Cho}
\address[S.~Cho]{Department of mathematics, Ajou University, 206, World cup-ro, Yeongtong-gu, Suwon 16499,  Republic of Korea}
\email{chosj@ajou.ac.kr}
\thanks{This work was supported by the Ajou University research fund.}

\author[S.~Choi]{Suyoung Choi}
\address[S.~Choi]{Department of mathematics, Ajou University, 206, World cup-ro, Yeongtong-gu, Suwon 16499,  Republic of Korea}
\email{schoi@ajou.ac.kr}
\thanks{The second named author was supported by Basic Science Research Program through the National Research Foundation of Korea(NRF) funded by the Ministry of Science, ICT \& Future Planning(NRF-2016R1D1A1A09917654).}
\author[S.~Kaji]{Shizuo Kaji}
\thanks{The third named author was partially supported by KAKENHI, Grant-in-Aid for Scientific Research (C) 18K03304.}
\address[S.~Kaji]{Institute of Mathematics for Industry, Kyushu University 744 Motooka, Nishi-ku, Fukuoka 819-0395, Japan}
\email{skaji@imi.kyushu-u.ac.jp}

\date{\today}

\begin{document}

\begin{abstract}
We develop a framework to construct geometric representations of finite groups $G$
through the correspondence between real toric spaces $X^\R$ and simplicial complexes with characteristic matrices.
We give a combinatorial description of the $G$-module structure of the homology of $X^\R$.
As applications, we make explicit computations of the Weyl group representations on the homology of
real toric varieties associated to the Weyl chambers
of type~$A$ and $B$, which show an interesting connection to the topology of posets.
We also realize a certain kind of Foulkes representation geometrically as the homology of real toric varieties.
\end{abstract}


	
\maketitle

\tableofcontents

\section{Introduction}
The fundamental theorem of toric geometry states that there is a bijective correspondence between the class of toric varieties of complex dimension $n$ and the class of fans (consisting of strongly convex simplicial cones) in $\R^n$.
This correspondence has opened up fertile research areas which reveal rich interactions between geometry and combinatorics.
For example, many studies on the cohomology of
toric varieties associated to root systems are found in the literature such as \cite{Stanley1989,Procesi1990,Dolgachev-Lunts1994,Stembridge1994,Abe2015}.

The real locus $X^\R$ of a toric variety $X$ of complex dimension $n$,
which is the fixed point set of the canonical involution,
forms a real subvariety of real dimension $n$, and it is called a \emph{real toric variety}.
While many arguments for complex toric varieties apply verbatim, it turns out that
real toric varieties sometimes exhibit essentially different nature.
In \cite{CKT} a generalization of real toric variety called real toric space is introduced.
In the present paper, we develop a framework to construct geometric representations of finite groups
through the correspondence between \emph{real toric spaces} and \emph{simplicial complexes with characteristic matrices}.
More precisely, given an action of a finite group $G$ on a simplicial complex,
we describe a necessary and sufficient condition with which the action induces one on the corresponding real toric space $X^\R$.
Furthermore, the $G$-module structure of the homology of $X^\R$ is described combinatorially (Theorem~\ref{stable-action}).

As an application, we investigate real toric varieties associated to
Weyl chambers in Section~\ref{sec:real_toric_var_corr_to_Weyl}.
The result demonstrates sharp difference from the complex case studied in \cite{Stanley1989,Procesi1990,Dolgachev-Lunts1994,Stembridge1994}.
In particular, we compute the Betti numbers for some exceptional cases, and we describe explicitly the $W_R$-module structures of $H_\ast(X^\R_R)$ for $R=A_n$ and $B_n$.
Remarkably, the real cases of type~$A$ and of type~$B$ reveal an interesting connection to the topology of posets studied in
\cite{Solomon1968,Orlik-Solomon1980,Stanley1982,Hanlon1984, Sundaram1994,Rains2010} (see also \cite{Wachs2007} for survey).
We note that the real toric varieties associated to the Weyl groups and their representations were already considered in \cite{Henderson-Lehrer2009, Henderson2012}.
We work on these objects from different viewpoints.
Especially, we provide simple combinatorial description for the representations of real toric varieties associated to Weyl groups of types~$A$ and $B$ as sums of irreducible representations, whereas the known results in \cite{Henderson2012} for type~$A$ form alternating sums.
The representation for type~$B$ is unknown before.

In Section~\ref{sec:nestohedra},
we investigate in the opposite direction; given a group representation, we construct a
real toric space with an action of the group whose homology carries the representation.
This is a typical situation in geometric representation theory, one of whose objectives is to seek for
geometric objects realising particular representations.
Precisely, here we realize a certain kind of Foulkes representation as the homology of real toric varieties.
Our theory provides yet another application of toric topology in representation theory.

As we are primarily interested in finite Weyl groups,
and all their irreducible representations over $\C$ are
realized over $\Q$ (see \cite{Benard1971} and the references therein),
the coefficients of (co)homology in this paper is always taken to be the rational numbers $\Q$ unless otherwise stated.

\section{Finite group action on real toric spaces}
Real toric space is a topological generalization of (simplicial) real toric variety.
In this section, we briefly review the definitions and notations regarding real toric spaces,
following \cite{CP2} and \cite{CKT}.
Then, we discuss the representation of a finite group $G$
on the homology of a real toric space, which is induced by a $G$-action on combinatorial data.

Let $K$ be a simplicial complex on $m$-vertices $[m]=\{1, \ldots, m\}$.
The \emph{real moment-angle complex} $\RZ_K$ of $K$ is defined as
\begin{align*}
    \RZ_K &= (\underline{D^1},\underline{S^0})^K \\
          &= \bigcup_{\sigma\in K} \left\{(x_1,\dotsc,x_m)\in (D^1)^m \mid x_i \in S^0 \text{ when }i\notin \sigma\right\},
\end{align*}
where $D^1=[-1,1]$ is the closed interval and $S^0 = \{-1,1\}$ is its boundary.

Let $\Lambda \colon \Z_2^m \to \Z_2^n$ be a linear map, where $\Z_2$ is the field with two elements.
We regard $\Lambda$ as an $n\times m$-matrix with elements in $\Z_2$,
and call it a \emph{mod $2$ characteristic matrix}.
The diagonal sign action of $\Z_2^m$ on $(D^1)^m$ induces an action of $\ker{\Lambda} \subset \Z_2^m$ on $\RZ_K$.
The \emph{real toric space} $M^\R(K,\Lambda)$ corresponding to the pair $K$ and $\Lambda$ is by definition the quotient space $\RZ_K/\ker{\Lambda}$.

Indeed, a real (simplicial) toric variety is a real toric space.
Let $X$ be a (simplicial) toric variety, and $X^\R$ its real toric variety. By the fundamental theorem of toric geometry, there is a (simplicial) fan $\Sigma$ corresponding to $X$. Let $V$ be the set of rays of $\Sigma$ and assume
$V=\{v_1,\ldots,v_m\}$.
Then, the collection of subsets of $V$ constituting the cones of $\Sigma$ form a simplicial complex $K$ on $V$.
In addition,
the primitive vectors in the direction of rays define a linear map $\Lambda \colon \Z_2^m \to \Z_2^n$,
whose $i$th column is the mod $2$ reduction of $v_i$.
Then, $X^\R$ is homeomorphic to the real toric space $\RZ_K/\ker{\Lambda}$.

Denote by $\row(\Lambda)$ the subspace of $\Z_2^m$ spanned by the row vectors of $\Lambda$;
that is, $\row(\Lambda)=\Im(\Lambda^T)=(\ker\Lambda)^\perp$, where $(\ker\Lambda)^\perp$
is the orthogonal complement to $\ker\Lambda$.
Notice that $\row(\Lambda)=\row(A\Lambda)$ and $\ker\Lambda=\ker(A\Lambda)$ for any invertible
$n\times n$-matrix $A$; that is, it is independent of the choice of the basis of $\Z_2^n$.
Throughout the paper, we identify the power set of $[m]$ with $\Z_2^m$ in the standard way;
each element of $[m]$ corresponds to a coordinate of $\Z_2^m$.

The group of simplicial automorphisms $\mathrm{Aut}(K)$ of $K$ acts on the real moment-angle complex $\RZ_K$,
which is extensively studied by \cite{Al-Raisi2014}.
Our first theorem considers finite group representations on the cohomology of
real toric spaces, which connects combinatorics, topology, and representation theory.
\begin{thm}\label{stable-action}
    Suppose $G$ is a subgroup of $\mathrm{Aut}(K)$ which satisfies
    one of the following two equivalent conditions
\begin{enumerate}
  \item $G$ preserves $\ker\Lambda$; that is, for any $g\in G$,
  $\ker\Lambda=\ker\Lambda P_g^{-1}$, where $P_g$ is the permutation matrix $\Z_2^m\to \Z_2^m$ corresponding to $g$;
  \item for each element $g\in G$, there exists an $n\times n$-matrix $A_g$ such that $\Lambda P_g^{-1}=A_g\Lambda$.
\end{enumerate}
  Then, the action of $G$ on $\RZ_K$ induces one on $M^\R(K,\Lambda)$
  and we have the following isomorphism of $G$-modules
 \begin{equation}\label{isom}
  H_\ast(M^\R(K,\Lambda))\simeq \bigoplus_{S \in \row \Lambda} \widetilde{H}_{\ast-1}(K_S).  \tag{$\ast$}
 \end{equation}
\end{thm}
\begin{proof}
We first prove the equivalence of two conditions (1) and (2).
The row space $\row(\Lambda)$ and $\ker\Lambda$ are the orthogonal complements to
each other in $\Z_2^m$. Since the $G$-action on $\Z_2^m$ is an isometry,
$\ker\Lambda$ is $G$-stable if and only if $\row(\Lambda)$ is $G$-stable.
This, in turn, is equivalent to that each row is mapped to a linear combination of row vectors.

Assume one of (1) and (2) is satisfied.
Since $M^\R(K,\Lambda)$ is the quotient $\RZ_K/\ker{\Lambda}$,
the action induces one on $M^\R(K,\Lambda)$, which in turn induces a $G$-module
structure on $H_\ast(M^\R(K,\Lambda))$.
Define a map $H$ as the composition
\[
H: \Sigma\RZ_{K}\longrightarrow\bigvee_{S\subset [m]}\Sigma\RZ_{K_{S}}
     \longrightarrow
      \bigvee_{S\subset [m]}\Sigma^{2}\vert K_{S}\vert,
      \]
where $K_S$ is the induced subcomplex of $K$ with respect to $S\subset [m]$.
The map $H$ is shown to be a homotopy equivalence in \cite{BBCG2010}. Moreover,
it is shown to be $G$-equivariant on the homology in \cite[Corollary 2.3.11]{Al-Raisi2014},
  where the action of $g\in G$ on the target space is induced by $g\colon K_S \to K_{gS}$.
Recall from \cite{CKT} the rational homotopy equivalence
\[
\psi\colon \Sigma\bigvee_{S\in \row(\Lambda)}\Sigma\vert K_{S}\vert \xrightarrow{\Sigma \iota}
\bigvee_{S\subset [m]}\Sigma^{2}\vert K_{S}\vert \xrightarrow{H^{-1}}
\Sigma\RZ_{K} \xrightarrow{\Sigma q}
\Sigma M^\R(K,\Lambda),
\]
where $\iota$ is the inclusion and $q$ is the quotient.
As all the maps appearing in the definition of $\psi$ are equivariant on the (co)homology,
$\psi$ induces an isomorphism of $G$-representations \eqref{isom}.
\end{proof}


This theorem provides a handy way to construct various geometric representations of finite groups,
which can be described explicitly in a combinatorial way as we will see in later sections.


\begin{ex}
Let $K$ be a $4$-gon on $\{1,2,3,4\}$, and $G$ a cyclic group of order $4$ generated by $g$.
Consider the $G$-action on $K$ defined by
$$g \cdot i = \left\{
                \begin{array}{ll}
                  i+1, & \hbox{$i=1,2,3$;} \\
                  1, & \hbox{$i=4$.}
                \end{array}
              \right.
$$
Put $\Lambda_1:=\begin{pmatrix}
  1 & 0 & 1 & 1 \\ 0 & 1 & 0 & 1
\end{pmatrix} $ and $\Lambda_2:=\begin{pmatrix}
  1 & 0 & 1 & 0 \\ 0 & 1 & 0 & 1
\end{pmatrix} $.
Then, one can see that $G$ does not preserve $\ker\Lambda_1$ because
$$
    g \cdot \Lambda_1 = \begin{pmatrix}
  1 & 1 & 0 & 1 \\ 1 & 0 & 1 & 0
\end{pmatrix} $$
has a different row space from that of $\Lambda_1$.
On the other hand, $G$ preserves $\ker\Lambda_2$ because
$$
    g \cdot \Lambda_2 = g^3 \cdot \Lambda_2 = \begin{pmatrix}
  0 & 1 & 0 & 1 \\ 1 & 0 & 1 & 0
\end{pmatrix}  = \begin{pmatrix}
  0 & 1 \\ 1 & 0
\end{pmatrix} \Lambda_2,
$$ and $g^2 \cdot \Lambda_2 = g^4 \cdot \Lambda_2 = \Lambda_2$.
This action on the corresponding real toric space $M^\R(K,\Lambda_2)\simeq T^2$ is seen to be $g\colon T^2 \to T^2,\ g(x,y)=(-y,x)$.
On the homology,
$H_*(M^\R(K,\Lambda_2))$ is the trivial representations on degree $0$ and $2$,
and the two dimensional irreducible representation (over $\Q$) on degree $1$.
On the other hand by \eqref{isom},
we have $H_1(M^\R(K,\Lambda_2))\simeq \tilde{H}_0(\{1,3\})\oplus \tilde{H}_0(\{2,4\})$.
Since $g\{1,3\}=\{2,4\}$ and $g\{2,4\}=\{3,1\}$,
taking account the sign of the induced map $g_*$ on the reduced homology,
we can also see that it is the two dimensional irreducible representation.
\end{ex}

\section{Real toric varieties associated to Weyl chambers} \label{sec:real_toric_var_corr_to_Weyl}
An illustrative family of examples of Theorem~\ref{stable-action}
are obtained by the real toric varieties associated to Weyl chambers.
The complex toric varieties associated to Weyl chambers are studied
by \cite{Procesi1990,Stanley1989,Stembridge1994,Dolgachev-Lunts1994},
and the real toric variety associated to the type~$A_n$ Weyl chambers is
studied by \cite[Corollary~1.2]{Henderson2012}, \cite{Suciu2012} and \cite{Choi-Park2015}.
Here, we carry out explicit computations of the representation of the Weyl groups on the homology,
which contrast sharply with the complex toric case.

We fix some notations as follows;
\begin{itemize}
\item $\Phi$: an irreducible root system of type~$R$ and rank $n$.
\item $\Delta=\{\alpha_1,\ldots, \alpha_n\}$: a fixed set of simple roots of $\Phi$.
\item $\Omega=\{\omega_1,\ldots,\omega_n\}$: the set of fundamental co-weights.
That is, $(\omega_i,\alpha_j)=\delta_{ij}$ with respect to the inner product in the ambient space.
\item $W_R=\langle s_1,\ldots, s_n \rangle$: the Weyl group generated by the simple reflections acting on $\R\langle \Omega \rangle$.
\item $V_R=W_R\cdot \Omega=\{v_1,\ldots,v_N\}$: the set of rays spanning the chambers.
\item $K_R\subset 2^{V_R}$: the corresponding simplicial complex. It is called the \emph{Coxeter complex} of type~$R$ \cite[\S 1.15]{Humphreys1990}.
\item $\Lambda_R=(v_1,v_2,\ldots,v_N)$: the mod 2 characteristic matrix, where the columns are
the mod 2 coordinates of the rays.
We set $v_1=\omega_1,\ldots,v_n=\omega_n$ so that the first $n\times n$-submatrix is the identity.
\item $X^\R_R$: the real toric space $M^\R(K_R, \Lambda_R)$; it is indeed the real toric variety associated to the Weyl chambers.
\end{itemize}
We recall the combinatorial structure of the Coxeter complex $K_R$. See
\cite{Bjorner1984} for more details.
\begin{itemize}
\item The set of $0$-simplices is the $W_R$-orbit of the fundamental co-weights $W_R \cdot \Omega$.
The orbits $W_R \omega_i$ and $W_R \omega_j$ are disjoint if $i\neq j$.
 The stabilizer of $\omega_i$ is the parabolic subgroup $W_i$ generated by $\{s_j \mid j\neq i\}$.
 Hence, the set of vertices is described by
 \[
 V_R = \{w\omega_i\mid  w\in W_R/W_i, \omega_i\in \Omega\}.
 \]
\item Two vertices $u$ and $v$ form a $1$-simplex if and only if there exist $ w\in W_R$ and $\omega_i,\omega_j\in \Omega$ such that $wu=\omega_i, wv=\omega_j$.
Similarly, the set of $(k-1)$-simplices is
\[
\{ w\cdot L\mid L\subset \Omega, |L|=k, w\in W_R/W_L \},
\]
where $W_L$ is the stabilizer of $L=\{\omega_{i_1},\ldots, \omega_{i_k}\}$,
which is the parabolic subgroup of $W_R$ generated by $\{s_i \mid i\not\in L\}$.
In particular, maximal simplices are obtained by the free and transitive action of $W_R$ on the fundamental chamber.
\end{itemize}
\medskip

By definition, the Weyl group $W_R$ acts on $K_R$.
We compute the action of $W_R$ on the rows of $\Lambda_R$ explicitly.
\begin{lem} \label{lem:Weyl_group_acts_well}
The Weyl group $W_R$ acts on $K_R$ and preserves $\ker\Lambda_R$.
More precisely,
let $\Lambda_R^j\in \Z_2^m$ be the $j$th row of $\Lambda_R$, which corresponds the $\omega_j$ coordinates of the rays.
Then, we have
\[
(s_i(\Lambda_R))^j=\Lambda_R^j-c_{ij}\Lambda_R^i,
\]
where $c_{ij}=(\alpha_i^\vee, \alpha_j)$ are the entries of the Cartan matrix of $R$.
\end{lem}
\begin{proof}
This is a special instance of Theorem~\ref{stable-action} since $W_R$ acts on $\R\langle \Omega \rangle$ linearly.
Recall that the action of the simple reflections on $\R\langle \Omega \rangle$ is given by
\[
s_i(\omega_j)=\omega_j-\dfrac{2(\alpha_i,\omega_j)}{(\alpha_i,\alpha_i)}\alpha_i
=\begin{cases} \omega_j & (i\neq j) \\ \omega_i-\alpha_i^\vee & (i=j) \end{cases},
\]
where $\alpha_i^\vee=\dfrac{2\alpha_i}{(\alpha_i,\alpha_i)}$ is the simple co-root.
Therefore, if we write $\alpha_i^\vee=\sum_j c_{ij} \omega_j$,
we have
\[
s_i(\sum_j d_j \omega_j) = \sum_j (d_j-d_ic_{ij})\omega_j.
\]
and $(s_i(\Lambda_R))^j=\Lambda_R^j-c_{ij}\Lambda_R^i\in \row(\Lambda_R)$.
\end{proof}

Now we state a theorem which gives a combinatorial description of the $W_R$ module structure
of the (co)homology of $X^\R_{R}$.
\begin{thm}\label{thm:Weyl_group_acts_well}
There is a $W_R$-module isomorphism
\[
    H^\ast(X^\R_{R})\cong H_\ast(X^\R_{R}) \cong \bigoplus_{S \in \row(\Lambda_{R})} \widetilde{H}_{\ast-1} (K_S).
\]
\end{thm}
\begin{proof}
Since irreducible characters of $W_R$-modules are real-valued for any Weyl group $W_R$,
$W_R$-modules are self-dual.
As $H^\ast(X^\R_{R})$ and $H_\ast(X^\R_{R})$ are dual to each other as $W_R$-modules,
we obtain the first isomorphism.
The second isomorphism follows by Lemma~\ref{lem:Weyl_group_acts_well} and Theorem~\ref{stable-action}.
\end{proof}

In the following subsections, by using Lemma~\ref{lem:Weyl_group_acts_well} and Theorem~\ref{thm:Weyl_group_acts_well}, we compute the Betti numbers of $X^\R_R$ for some exceptional cases, and describe explicitly the $W_R$-module structures of $H_\ast(X^\R_R)$
for $R=A_n$ and $B_n$.
In particular, $H_\ast(X^\R_{A_n})$ is isomorphic to the famous Foulkes representation.

\subsection{Computation of Betti numbers}
As an interesting application of this explicit description, 
we compute the Betti numbers of $X^\R_R$ for some exceptional cases.
The (rational) Betti number of $X^\R_R$ has been computed when $R=A_n$ (in \cite{Henderson2012}, \cite{Suciu2012} and \cite{Choi-Park2015}),
$R=B_n$ (in \cite{CPP16}), and when $R=C_n$ and $D_n$ (in \cite{CKP}).
However, for exceptional types the Betti numbers have not been known except for $R=G_2$. 
It is known that $X^\R_{G_2}$ is homeomorphic to the connected sum of ten copies of the real projective plane.
In this subsection, we successfully apply our result to compute the Betti numbers for the cases
 $R=F_4$ and $E_6$.

In order to compute the Betti number of $X^\R_R$, it is enough to compute the rank of $H_\ast(K_S)$ for all $S \in \row(\Lambda_R)$ due to \eqref{isom}.
By Lemma~\ref{lem:Weyl_group_acts_well},
$K_S \simeq K_{gS}$ for $S\in \row(\Lambda_R)$ and $g\in W_R$.
Therefore,
we only have to consider representatives $K_S$ of the $W_R$-orbits in $\row(\Lambda_R)$,
which can be obtained by the explicit description of the $W_R$-action on the rows of $\Lambda$ given in Lemma~\ref{lem:Weyl_group_acts_well}.
The reduced Betti number $\widetilde{\beta}_\ast(K_S)$ of $K_S$ is computable by a computer program
 (for example, CHomP \cite{chomp}) within reasonable time for $R$ of small rank.
By exploiting this symmetry, we obtained the following result with the aid of computer
(the Maple code used to produce the result is available at \url{https://github.com/shizuo-kaji/WeylGroup}).
\begin{prop} \label{prop:bettinumber_F4_E6}
The Betti numbers of $X^\R_R$ for $R=G_2,F_4,E_6$ are given as follows:
    $$ \beta_i(X^\R_{G_2}) = \left\{
                                 \begin{array}{ll}
                                   1, & \hbox{$i=0$;} \\
                                   9, & \hbox{$i=1$;} \\
                                   0, & \hbox{otherwise,}
                                 \end{array}
                               \right. \quad
\beta_i(X^\R_{F_4}) = \left\{
                                 \begin{array}{ll}
                                   1, & \hbox{$i=0$;} \\
                                   57, & \hbox{$i=1$;} \\
                                   264, & \hbox{$i=2$;} \\
                                   0, & \hbox{otherwise,}
                                 \end{array}
                               \right.
    \quad \text{ and }
$$
$$
\beta_i(X^\R_{E_6}) = \left\{
                                 \begin{array}{ll}
                                   1, & \hbox{$i=0$;} \\
                                   36, & \hbox{$i=1$;} \\
                                   1323, & \hbox{$i=2$;} \\
                                   4392, & \hbox{$i=3$;} \\
                                   0, & \hbox{otherwise.}
                                 \end{array}
                               \right.
    $$
\end{prop}
\begin{proof}
For type~$G_2$, there are three non-zero elements in $\row(\Lambda_{G_2})$.
These are all in the same orbit, and, for the first row $S$ of $\Lambda_{G_2}$, $K_S$ consists of $4$ distinct points,
and hence, $\widetilde{\beta}_0(K_S) = 3$.
Therefore, the first reduced Betti number $\widetilde{\beta}_1(X^\R_{G_2})$ of $X^\R_{G_2}$ is $3 \times 3 = 9$.

For type~$F_4$ there are $15$ non-zero elements in $\row(\Lambda_{F_4})$ and they are partitioned into two orbit types;
$12$ elements are in the orbit generated by the first row of $\Lambda_{F_4}$
($\widetilde{\beta}_0 (K_S)=1, \widetilde{\beta}_1(K_S)=22$)
and $3$ elements are in the orbit generated by the third row ($\widetilde{\beta}_0(K_{S})=15$).

For type~$E_6$ there are $63$ non-zero elements in $\row(\Lambda_{E_6})$. They are partitioned into two orbit types; $27$ elements are in the orbit generated by the first row ($\widetilde{\beta}_1(K_{S})=49$) and $36$ elements are in the orbit generated by the sum of the first and the fourth row ($\widetilde{\beta}_0(K_{S})=1, \widetilde{\beta}_2(K_{S})=122$).
\end{proof}

\begin{rem} \label{rem:h-vector_of_exceptional}
  It is well-known the $i$th $\Z_2$-Betti number of a real toric variety $X^\R$ is equal to the $i$th component of the $h$-vector of its underlying complex $K$ (see \cite{Davis-Januszkiewicz1991}). We denote by $h_q(K)$ the $h$-polynomial of a simplicial complex $K$, and by $\chi(X^\R)$ the Euler characteristic number of $X^\R$.
  Here is the list of the $h$-polynomials of $K_R$ for exceptional types $R$:
  \begin{align*}
    h_q(K_{G_2}) =& q^2+10q+1, \\
    h_q(K_{F_4}) =& q^4+236q^3+678q^2+236q+1,  \\
    h_q(K_{E_6}) =&q^6+1272 q^5+12183q^4+24928q^3+12183q^2+1272q+1,  \\
    h_q(K_{E_7}) =&q^7+17635q^6+309969q^5+1123915q^4+1123915q^3+309969q^2+17635q+1, \\
    h_q(K_{E_8}) =&q^8+881752q^7+28336348q^6+169022824q^5+300247750q^4+169022824q^3 \\ & +28336348q^2+881752q+1.
  \end{align*}
  We recall that the Euler characteristic does not depend on the field with respect to which cohomology is taken. Therefore, by using
  above, we can compute the Euler characteristic numbers as
  \[
   \chi(X^\R_{G_2}) = -8, \chi(X^\R_{F_4}) = 208, \chi(X^\R_{E_6}) = -3104,
  \chi(X^\R_{E_7}) = 0, \text{ and } \chi(X^\R_{E_8}) = 17111296.
  \]
  The Euler characteristic numbers of $X^\R_{G_2},X^\R_{F_4}$, and $X^\R_{E_6}$ can be also computed from Proposition~\ref{prop:bettinumber_F4_E6}.

For type~$E_7$, there are three orbit types; $63$ elements are in the orbit generated by the first row of $\Lambda_{E_7}$, $63$ elements are in the orbit generated by the sum of the fifth and the seventh rows, and one element in the orbit generated by the sum of the first, the fifth and the seventh rows.
For type~$E_8$, there are two orbit types; $120$ elements are in the orbit generated by the first row of $\Lambda_{E_8}$, $135$ elements are in the orbit generated by the sum of the sixth and the eighth rows.
However, we are currently unable to compute the reduced Betti numbers of $K_S$ since it is too big.
\end{rem}

\begin{rem}
  It should be noted that $K_S$ also determines the torsions in (co)homology except for $2$-torsion. More precisely, by \cite{CP2} and \cite{Cai-Choi_even_torsion}, we know that, for $k>1$ and odd number $q$,
  \begin{align*}
H^\ast(M(K,\Lambda);\Z_{2^k}) \cong& \bigoplus_{ S \in \row( \Lambda )} \widetilde{H}^{\ast-1} (K_S; \Z_{2^{k-1}}), \quad \text{ and }\\
H^\ast(M(K,\Lambda);\Z_q) \cong& \bigoplus_{ S \in \row( \Lambda )} \widetilde{H}^{\ast-1} (K_S; \Z_q).
  \end{align*}
  With the help of computer program, one can easily check that each $K_S$ in Proposition~\ref{prop:bettinumber_F4_E6} is torsion-free, so $X^\R_{R}$ for $R=G_2, F_4, E_6$ has no torsion in (co)homology other than $2$-torsion. Therefore, by Proposition~\ref{prop:bettinumber_F4_E6} and Remark~\ref{rem:h-vector_of_exceptional}, one can apply the universal coefficient theorem to compute the integral (co)homology group for $X^\R_{R}$ for $R=G_2, F_4, E_6$.
\end{rem}

\subsection{Type $A_n$ representation} \label{subsec:Type_A}
When $R=A_n$, the Weyl group $W_{A_n}$ is the symmetric group $\Sym_{n+1}$
and $K_{A_n}$ has a nice description as the dual of the permutohedron of order $n+1$.
 The map $V_{A_n}\to 2^{[n+1]}\setminus\{ [n+1], \emptyset \}$ sending $w\omega_i\to \{w(1),w(2),\ldots,w(i)\}$
 induces an $\Sym_{n+1}$-equivariant bijection.
 Under this identification, $J_1,\ldots,J_k\subset [n+1]$ form a simplex if and only if
 they form a nested chain of subsets up to permutation.
 To represent $\Lambda_{A_n}$, we use the basis consisting of $\e_k:=(1,k)\omega_1$ for $1\le k\le n$, where $(1,k)\in \Sym_{n+1}$ is the transposition.
 The coordinates of $J\in V_{A_n}$ with this basis are given by $\sum_{k\in J} \e_k$, where we set $\e_{n+1}=\sum_{i=1}^n\e_i$.

Note that Theorem~\ref{thm:Weyl_group_acts_well}
 specializes in this case to
giving an $\Sym_{n+1}$-module isomorphism $H_\ast(X^\R_{A_n}) \cong \bigoplus_{S \in \row(\Lambda_{A_n})} \widetilde{H}_{\ast-1} (K_S)$.
Thus, in order to compute $H_\ast(X^\R_{A_n})$, we have to investigate each $K_S$.
We establish a correspondence between
$\row(\Lambda_{A_n})$ and the even cardinality subsets of $[n+1]$ in the following way.
Any element $S\in \row(\Lambda_{A_n})$ is a linear sum of rows of $\Lambda_{A_n}$.
If it is a sum of an even number of rows, say $i_1,i_2,\ldots,i_{2r}$th rows,
then we associate $I_S =\{i_1,i_2,\ldots,i_{2r}\}\subset [n+1]$.
If it is a sum of an odd number of rows, say $i_1,i_2,\ldots,i_{2r-1}$th rows,
then we associate $I_S =\{i_1,i_2,\ldots,i_{2r-1}, n+1\}\subset [n+1]$.

One can see that $K_S$ is the full-subcomplex of $K_{A_n}$ consisting of the vertices $J$ such that $|J \cap I_S|$ is odd. By \cite[Lemma~5.3]{Choi-Park2015},
one can inductively remove $J \not\subset I_S$ from $K_S$ without changing the homotopy type of $K_S$.
This gives a homotopy equivalence between $K_S$ and the poset complex associated to
the odd rank-selected Boolean algebra $\cB_{I_S}^{odd}$ over $I_S$,
which induces an $\Sym_{n+1}$-module isomorphism
\begin{equation}\label{homology_A}
H_*(X^\R_{A_n}) \cong \bigoplus_{I\in {[n+1]\choose 2r}} \widetilde{H}_{*-1} (\cB_{I}^{odd}).
\end{equation}
Before we describe the $\Sym_{n+1}$-module structure of $H_*(X^\R_{A_n})$ in detail, we review some basic facts on the representations of symmetric groups. A reference for basics of the representation theory of symmetric groups is \cite{Sagan2001book}. A \emph{partition} $\lambda$ of a positive integer $k$, denoted by $\lambda \vdash k$, is a non-increasing sequence $\lambda=(\lambda_1, \dots, \lambda_\ell)$ of positive integers such that $\sum_i \lambda_i=k$ and we let $|\lambda|=k$. The \emph{diagram} of a partition $\lambda$ is the left-justified array of $\lambda_i$ boxes in the $i$th row, and a \emph{standard tableau} of shape $\lambda$ is a filling of boxes in the diagram of $\lambda$ with numbers $1, 2, \dots, |\lambda|$ whose rows are increasing from left to right and columns are increasing from top to bottom. The number of standard tableaux of shape $\lambda$ is denoted by $f^\lambda$. For two partitions $\lambda,\mu$ such that $\mu_i\leq \lambda_i$ for all $i$, $\mu\subseteq \lambda$ in notation, the diagram of $\lambda/\mu$ is the diagram of $\lambda$ with the first $\mu_i$ boxes deleted in the $i$th row for all $i$.
An entry $i$ of a standard tableau is a \emph{descent} if $i+1$ is located in a lower row than the row where $i$ is located. The set of descents of a standard tableau $T$ is called the \emph{descent set of $T$}.

Irreducible representations of the symmetric group $\Sym_{k}$ are indexed by the partitions of $k$. For a partition  $\lambda=(\lambda_1, \dots, \lambda_\ell)$ of $k$, corresponding irreducible $\Sym_k$-module $S^\lambda$ is called the \emph{Specht module}, and it is known that
$\dim S^\lambda = f^\lambda$.

Each direct summand on the right hand side of \eqref{homology_A} was computed by Solomon and also by Stanley:

\begin{thm}[{\cite{Solomon1968,Stanley1982}}]\label{thm:skew_Specht}
Let $Q \subset [m-1]$. Then
the homology of the $Q$-rank-selected poset $\cB_{[m]}^{Q}$ is given, as an $\Sym_m$-module, by
\[
\widetilde{H}_\ast(\cB_{[m]}^{Q})\cong \begin{cases} {\displaystyle \bigoplus_{\nu} c_{Q,\nu} S^\nu} & (\ast=|Q|-1) \\ 0 & (\ast \neq |Q|-1), \end{cases}
\]
where $c_{Q,\nu}$ is the number of standard tableaux of shape $\nu$ with  descent set $Q$.

In particular,
by setting $Q=\{1,3,\ldots, 2r-1\}$,
 we have an $\Sym_{[2r]}$-module isomorphism
\[
\widetilde{H}_\ast(\cB_{[2r]}^{odd})\cong \begin{cases} {\displaystyle \bigoplus_{\nu} c_{Q,\nu} S^\nu} & (\ast=r-1) \\ 0 & (\ast \neq r-1). \end{cases}
\]
\end{thm}

\begin{rem}[{See \cite[Theorem 3.4.4]{Wachs2007}}]\label{rmk:skew_Specht}
The module $\bigoplus_{\nu} c_{Q,\nu} S^\nu$ that appears in Theorem~\ref{thm:skew_Specht} can be understood as
a \emph{skew Specht module} $S^{\lambda/\mu}$ for some skew hook $\lambda/\mu$.
A skew hook is an edge-wise connected skew Young diagram which contains no $(2\times 2)$-subdiagram.
A skew hook with $k$ boxes is determined by its \emph{descent set} $S\subset [k-1]$;
when one fills a skew hook with numbers $1$ through $k$ in order from southwest to northeast, $i$ is a \emph{descent} of the skew hook if $i+1$ is placed above $i$.
A skew Specht module corresponding to a skew hook is often referred to as the \emph{Foulkes representation}.
Its decomposition into irreducible modules (Specht modules of ordinary shape) follows from the Littlewood-Richardson rule,
and $c_{Q,\nu}$ is seen to be the Littlewood-Richardson number $c_{\mu \nu}^\lambda$.
See \cite{Sagan2001book} for the description of Littlewood-Richardson numbers.

In particular,
\[
\widetilde{H}_{r-1}(\cB_{[2r]}^{odd})\cong  S^{\lambda_r/\mu_r}
\] when $\lambda_r=(r, r, r-1, r-2, \dots, 2, 1)$, and $\mu_r=(r-1, r-2, \dots, 2, 1)$. For example, the diagram of the skew hook $\lambda_2/\mu_2$ is\quad   $\tableau{
        & \hspace{.3cm} \\
        \hspace{.3cm}  & \hspace{.3cm}  \\
         \hspace{.3cm}&
   }$\,\,,
whose descent set is $\{1,3\}$.
\end{rem}
\medskip

Let $\Sym_I$ be the stabilizer subgroup of $\Sym_{n+1}$ fixing all $i\not\in I$.
Then $\Sym_I$ preserves $\widetilde{H}_*(\cB_{I}^{odd})$
and $\Sym_{\bar{I}}$ acts on $\widetilde{H}_*(\cB_{I}^{odd})$ trivially, where $\bar{I}=[n+1] \setminus I$.
Moreover, each $I\in {[n+1]\choose 2r}$ is in one to one correspondence with a coset $\sigma (\Sym_{\{1, \dots, 2r\}} \times \Sym_{\{2r+1, \dots, n+1\}})$ in $\Sym_{n+1}$ via $I=\sigma \{1, \dots, 2r\}$. This proves the following.

\begin{thm}\label{thm:induced}
The $r$th homology $H_r(X^\R_{A_n})$ of $X^\R_{A_n}$ with the natural action of $\Sym_{n+1}$, is isomorphic to the induced representation
${\rm Ind}_{ \Sym_{\{1, \dots, 2r\}} \times \Sym_{\{2r+1, \dots, n+1\}}}^{\Sym_{n+1}}(\widetilde{H}_{r-1}(\cB_{[2r]}^{odd})\otimes S^{(n-2r)})$ of $\Sym_{n+1}$,
where $S^{(n-2r)}$ is the trivial representation of $\Sym_{\{2r+1, \dots, n+1\}}$.
\end{thm}

The irreducible decomposition is given by a standard argument:

\begin{cor}\label{cor:typeA}
Let $Q=\{1, 3, \dots, 2r-1\}$ and let $c_{Q, \nu}$ be the number of standard tableaux of shape $\nu$ with descent set $Q$. Then, we have
\[
H_r(X^\R_{A_n})\cong \bigoplus_{\eta \vdash (n+1) } \left(\sum_{\nu} c_{Q,\nu}\right) S^\eta,
\]
where $\nu$ runs over all partitions of $2r$ that are contained in $\eta$, and $\eta/\nu$ has at most one box in each column.
\end{cor}

\begin{proof}
By Theorem~\ref{thm:induced}, we have following $\Sym_{n+1}$ module isomorphisms.
\begin{align*} H_r(X^\R_{A_n})&\cong {\rm Ind}_{ \Sym_{\{1, \dots, 2r\}} \times \Sym_{\{2r+1, \dots, n+1\}}}^{\Sym_{n+1}}\left(\bigoplus_{\nu} c_{Q,\nu} S^\nu \otimes S^{(n-2r)}\right) \\
&\cong \bigoplus_{\nu}  c_{Q,\nu} \left( {\rm Ind}_{ \Sym_{\{1, \dots, 2r\}} \times \Sym_{\{2r+1, \dots, n+1\}}}^{\Sym_{n+1}}\left( S^\nu \otimes S^{(n-2r)}\right)\right)\\
&\cong \bigoplus_{\nu} c_{Q,\nu} \left(\bigoplus_{\nu\leadsto \eta} S^\eta\right)
 \cong \bigoplus_{\eta\vdash (n+1) } \left(\sum_{\nu\leadsto \eta} c_{Q,\nu}\right) S^\eta
\end{align*}
where $\nu\leadsto \eta$ means that we can obtain $\eta$ from $\nu$ by adding $n+1-2r$ boxes so that no two new boxes are in one column. In the third isomorphism, the well known Pieri rule (a special case of the Littlewood-Richardson rule) is used.
\end{proof}

\begin{ex}
If $n=5$ and $r=3$, we let $Q=\{1, 3, 5\}$. Then, the homology $H_3(X^\R_{A_5})$ is decomposed  into irreducible $\Sym_6$ modules in the following way due to Corollary~\ref{cor:typeA}.
 $$H_3(X^\R_{A_5})\cong S^{(3,3)}\oplus 2 S^{(3,2,1)}\oplus S^{(3,1,1,1)}\oplus S^{(2,2,2)}\oplus S^{(2,2,1,1)}\,.$$

For example, when $\eta=(2, 2, 2)\vdash 6$, the only $\nu$ we have to consider is $(2, 2, 2)$ and $c_{Q,\nu}=1$ hence $S^{(2,2,2)}$ is a summand in the decomposition of $H_3(X^\R_{A_5})$ with multiplicity $1$.
We list all standard tableaux of shape $\nu$ with descent set $Q$ for all possible $\nu\vdash 6$ .

$$  \tableau{
       1 & 3 &5\\
       2  & 4&6  }
 \quad
 \quad
 \tableau{
        1 & 3 &5\\
       2  & 4 \\
       6 }
 \quad
 \quad
  \tableau{
        1 & 3 &5\\
       2  & 6 \\
       4 }
\quad
\quad
       \tableau{
       1 & 3& 5 \\
       2   \\
       4 \\
       6 }
\quad
\quad
       \tableau{
       1 & 3 \\
       2  & 5  \\
       4 &6 }
\quad
\quad
       \tableau{
       1 & 3 \\
       2  & 5  \\
       4\\
       6 }
$$
It can be checked that $$\rm{dim}(S^{(3,3)})+ 2 \rm{dim}(S^{(3,2,1)})+\rm{dim}(S^{(3,1,1,1)})+\rm{dim} (S^{(2,2,2)})+\rm{dim}(S^{(2,2,1,1)})=5+2\times16+10+5+9=61\,,$$
which is the number of alternating permutations of length $6$ (times ${n+1}\choose {2r}$), and it is known to be the dimension of $H_3(X^\R_{A_5})$ (see \cite{CPP16,Henderson2012}).
\end{ex}

\begin{ex} When $n=5$ and $r=2$, we have the decomposition of $H_2(X^\R_{A_5})$ as in the following; $$H_2(X^\R_{A_5})\cong S^{(4, 2)}\oplus S^{(4, 1, 1)}\oplus 2 S^{(3,2,1)}\oplus S^{(3,1,1,1)}\oplus S^{(2,2,2)}\oplus S^{(2,2,1,1)} \,.$$
To compute the multiplicity of $S^{(3, 2, 1)}$, we consider partitions of $2r=4$, which is contained in $(3,2,1)$ and can be obtained from $(3, 2, 1)$ by deleting two boxes from different columns. There are three such partitions $(2, 1, 1), (2, 2)$ and $(3, 1)$, and the values $c_{Q, \nu}$ for each of them are $1, 1, 0$ respectively where $Q=\{1, 3\}$. Hence $S^{(3, 2, 1)}$ appears twice in the decomposition of $H_2(X^\R_{A_5})$.
By computing the dimension of each Specht module, we have
\begin{align*} \mathrm{dim}(H_2(X^\R_{A_5}))
=9+10+2\times16+10+5+9 =75\,.
\end{align*}
Note that $75=5\times 15$, where $5$ is the number of alternating permutations of length $4$, and $15$ is the index of the parabolic subgroup $\Sym_4\times\Sym_2$ in $\Sym_6$.
\end{ex}

\begin{rem} We, in Corollary~\ref{cor:typeA}, show how to write down the rational homology $H_r(X_{A_n}^\mathbb R)$  as a sum of irreducible $\Sym_{n+1}$-modules in an explicit way, while Henderson in his 2012 paper showed that the $\Sym_{n+1}$-module structure of the rational cohomology   $H^r(X_{A_n}^\R)$ is isomorphic to a signed sum of induced representations of the tensor product of a trivial representation and a sign representation of some parabolic subgroups of $\Sym_{n+1}$; see Corollary~1.2  in \cite{Henderson2012}.
\end{rem}

\subsection{Type $B_n$ representation}
When $R=B_n$, the Weyl group $W_{B_n}$ is the group of signed permutations (also called the hyperoctahedral group) and $K_{B_n}$ has a nice description as the dual of the type~$B_n$ permutohedron.
We write $[\pm n]=\{\pm1, \ldots, \pm n\}$, $-J=\{-j \mid j\in J\}$ for $J\subset [\pm n]$,
and $T_n=\{J\in 2^{[\pm n]}\mid J \cap -J=\emptyset, J\neq \emptyset\}$.
For each $J \in T_n$,
we write $J^{\pm} = (J\cup -J) \cap [n]$.

 The map $V_{B_n}\to T_n$ sending $w\omega_i\to \{w(1),w(2),\ldots,w(i)\}$
 induces a $W_{B_n}$-equivariant bijection.
 Under this identification, $J_1,\ldots,J_k\in T_n$ form a simplex if and only if
 they form a nested chain of subsets up to permutation.
 To represent $\Lambda_{B_n}$, we use the basis consisting of $\e_k:=(1,k)\omega_1$ for $1\le k\le n$, where $(1,k)\in W_{B_n}$ is the transposition.
 The coordinates of $J\in V_{B_n}$ with this basis are given by
 $\sum_{k\in J^{\pm}} \e_k$.
Note that, by Theorem~\ref{stable-action}, $H_\ast(X^\R_{B_n}) \cong \bigoplus_{S \in \row(\Lambda_{B_n})} \widetilde{H}_{\ast-1} (K_S)$ as $W_{B_n}$-modules.
We establish a correspondence between $\row(\Lambda_{B_n})$ and the power set of $[n]$ in the following way.
If $S\in \row(\Lambda_{B_n})$ is a sum of $r$ rows, say $i_1,\ldots,i_{r}$th rows,
then we associate $I_S =\{i_1,\ldots,i_{r}\}\subset [n]$ to it.

One can see that $K_S$ is the full-subcomplex of $K_{B_n}$ consisting of the vertices $J$ such that $|J^{\pm} \cap I_S|$ is odd.
By Lemmas~3.7 and 3.8 in \cite{CPP16}, all vertices $J$ such that $J^{\pm} \not\subset I_S$ are removable without changing the homotopy type of $K_S$.
Therefore, $K_S$ is homotopy equivalent to the poset complex associated to
the odd rank-selected lattice $\mathcal{C}_{I_S}^{odd}$ of faces of the cross-polytope over $I_S$.

We review some basic facts on the representations of hyperoctahedral groups. (See \cite{Geissinger1978}).
A \emph{double partition} $(\lambda, \mu)$ of $n$, denoted by $(\lambda,\mu)\vdash n$, is an ordered pair of partitions such that $|\lambda|+|\mu|=n$.

 Irreducible representations of the hyperoctahedral group $W_{B_n}$ are indexed by the double partitions of $n$ and we let $S^{(\lambda, \mu)}$ be an irreducible representation corresponding to the double partition $(\lambda, \mu)$. Then the dimension of $S^{(\lambda, \mu)}$ is given by the  number of \emph{double standard tableaux} of shape $(\lambda, \mu)$, that is the ordered pair of standard tableaux of shape $\lambda$ and $\mu$ respectively such that each number from 1 to $n$ appears exactly once in two tableaux. Hence $\dim(S^{(\lambda, \mu)})={n\choose k} f^\lambda f^\mu$, where $\lambda\vdash k$.
For a double standard tableau $(T_1, T_2)$, an entry $i$ is a \emph{descent} if

\begin{enumerate}
  \item $i$ and $i+1$ are in the same tableau and $i+1$ appears in a lower row than $i$, or
  \item $i$ is in $T_1$, $i+1$ is in $T_2$, or
  \item $i=n$ appears in $T_1$.
\end{enumerate}

Stanley computed the $W_{B_r}$-module $\widetilde{H}_{*} (\mathcal{C}_{I_S}^{odd})$:
\begin{thm}[{\cite{Stanley1982}}]\label{thm:stanley1982}
When $|I_S|=r$,
\[
\widetilde{H}_{*} (\mathcal{C}_{I_S}^{odd})\cong
\begin{cases} {\displaystyle \bigoplus_{(\lambda,\mu)\vdash r} b(\lambda,\mu) S^{(\lambda, \mu)}} & (*=\lfloor \frac{r-1}{2} \rfloor) \\
 0 & (*\neq \lfloor \frac{r-1}{2} \rfloor)\end{cases},
\]
where $b(\lambda,\mu)$ is the number of double standard Young tableaux of shape $(\lambda,\mu)$ whose descent set is
the set of odd numbers less than or equal to $r=|\lambda|+|\mu|$.
\end{thm}

Similarly to Theorem \ref{thm:induced}, we can prove the following theorem:
\begin{thm}\label{thm:induced_typeB}
The $k$th homology $H_k(X^\R_{B_n})$ of $X^\R_{B_n}$ with the natural action of $W_{B_n}$ is isomorphic to the sum of two induced representations
\[\bigoplus_{r\in \{2k-1, 2k\}}\left({\rm Ind}_{ W_{B_r} \times W_{B_{n-r}}}^{W_{B_n}}\left(\bigoplus_{(\lambda,\mu)\vdash r} b(\lambda,\mu)S^{(\lambda,\mu)}\otimes S^{(\emptyset; (n-r))}\right)\right)\] of $W_{B_n}$,
where $S^{(\emptyset; (n-r))}$ is the trivial representation of $W_{B_{n-r}}$ and the embedding  $W_{B_r} \times W_{B_{n-r}}\subseteq W_{B_n}$ comes from the diagonal embedding $GL_{r}(\mathbb R)\times GL_{n-r}(\mathbb R)\subseteq GL_{n}(\mathbb R)$.
\end{thm}

The irreducible decomposition is given by a standard argument:

\begin{cor}\label{cor:typeB}
We have
\[
H_{k}(X^\R_{B_n})\cong \bigoplus_{(\lambda,\nu)\vdash n} \left(\sum_{r\in \{2k-1, 2k\}} \sum_{
\mu} b(\lambda, \mu)\right) S^{(\lambda, \nu)}\,,
\]
where $\mu$ in the inside summation, runs over all partitions that are contained in $\nu$, and $\nu/\mu$ has $(n-r)$ boxes with at most one box in each column.
\end{cor}

\begin{proof} We compute a summand of the decomposition given in Theorem~\ref{thm:induced_typeB}  using (type B) Littlewood-Richardson rule. Then the result is a direct consequence.

We have the following $W_{B_n}$-module isomorphisms:
\begin{align*}
&{\rm Ind}_{ W_{B_r} \times W_{B_{n-r}}}^{W_{B_n}}\left(\bigoplus_{(\lambda,\mu)\vdash r} b(\lambda,\mu)S^{(\lambda,\mu)}\otimes S^{(\emptyset; (n-r))}\right)\\
&\cong  \bigoplus_{(\lambda,\mu)\vdash r} b(\lambda, \mu)\,\, \left({\rm Ind}_{ W_{B_r} \times W_{B_{n-r}}}^{W_{B_n}} (S^{(\lambda,\mu)}\otimes S^{(\emptyset; (n-r))})\right)\\
&\cong \bigoplus_{(\lambda,\mu)\vdash r} b(\lambda, \mu) \left(\bigoplus_{(\nu_1, \nu_2)} c_{\lambda, \emptyset}^{\nu_1} c_{\mu, (n-r)}^{\nu_2} S^{(\nu_1, \nu_2)}\right)\\
&\cong \bigoplus_{(\lambda,\mu)\vdash r} b(\lambda, \mu)\left(\bigoplus_{\mu\leadsto  \nu} S^{(\lambda, \nu)}\right)
\cong \bigoplus_{(\lambda,\nu)\vdash n} \left(\sum_{\mu\leadsto \nu} b(\lambda, \mu)\right) S^{(\lambda, \nu)}\,.
\end{align*}
where $\mu\leadsto \nu$ means that we can obtain $\nu$ from $\mu$ by adding $n-r$ boxes so that no two new boxes are in one column. In the second isomorphism, we used the Littlewood-Richardson rule for hyperoctahedral group representations (see Chapter 6 of \cite{GP2000book}), where $c_{\alpha, \beta}^{\gamma}$ is the well known Littlewood-Richardson number.
\end{proof}

\begin{ex} We let $n=3$ and $k=2$, then since $2k=4>n$ we need take care of the case when $r=3$ only. Among 10 possible double partitions $(\lambda, \mu)\vdash 3$, only four of them have nonzero $b(\lambda, \mu)$; $b((1), (1, 1))=b((2), (1))=b((1, 1), (1))=b((2, 1), \emptyset)=1$. For example, $ \left(\, \,\,\tableau{ 1\\3}\,\,\,,\,\, \tableau 2 \,\,\right)$ is the double standard tableau of shape $((1, 1), (1))$ with descent set $\{1, 3\}$. We hence have the following decomposition of $ H_{2}(X^\R_{B_3})$ into irreducible $W_{B_3}$ modules;
\[H_{2}(X^\R_{B_3})\cong S^{((1), (1, 1))}\oplus S^{((2), (1))}\oplus S^{((1,1), (1))} \oplus S^{((2, 1), \emptyset)}\,.\]
Since ${\dim}(S^{((1), (1, 1))})={\dim}(S^{((2), (1))})={\dim}(S^{((1,1), (1))})=3$ and ${\dim}(S^{((2, 1), \emptyset)})=2$, ${\dim}(H_{2}(X^\R_{B_3}))=11$, which is the number of alternating signed permutations in $W_{B_3}$ (see \cite{CPP16}).
\end{ex}

\begin{ex} Due to Corollary~\ref{cor:typeB}, we have the following decomposition;
\[H_{1}(X^\R_{B_3})\cong S^{(\emptyset, (2, 1))} \oplus S^{(\emptyset, (1,1, 1))}\oplus 2 S^{((1), (2))} \oplus S^{((1), (1, 1))}\,.\]
Since ${\dim}(S^{(\emptyset, (2, 1))})=2$, ${\dim}(S^{(\emptyset, (1, 1, 1))})=1$, ${\dim}(S^{((1), (2))})={\dim}(S^{((1), (1, 1))})=3$, we have  \[{\dim}(H_{1}(X^\R_{B_3}))=12={3\choose 2}b_2+{3\choose 1}b_1,\] where $b_i$ is the number of alternating signed permutations in $W_{B_i}$, as it was shown in \cite{CPP16}.
\end{ex}

\section{Real toric varieties associated to nestohedra}\label{sec:nestohedra}
In this section, we deal with real toric varieties corresponding to nestohedra as a generalization of $X^\R_{A_n}$ introduced in Section~\ref{sec:real_toric_var_corr_to_Weyl}.
Then, we realize some Foulkes representations as the top homology of those varieties.

A \emph{building set} $\fB$ on a finite set $S$ is a collection of nonempty subsets of $S$ such that
    \begin{enumerate}
        \item $\fB$ contains all singletons $\{i\}$, $i\in S$,
        \item if $I, J \in \fB$ and $I\cap J\neq\varnothing$, then $I\cup J\in\fB$.
    \end{enumerate}
Let $\fB$ be a building set on $[n+1]=\{1,\ldots,n+1\}$.
If $[n+1] \in \fB$, then $\fB$ is said to be \emph{connected}.

\begin{dfn}
    For a connected building set $\fB$ on $[n+1]$, a subset $N\subset \fB\setminus\{[n+1]\}$ is called a \emph{nested set}
    if both of the following conditions hold:
    \begin{enumerate}[{(N}1)]
        \item For any $I, J\in N$, one has $I\subset J$, $J \subset I$, or $I\cap J=\emptyset$.\label{n1}
        \item For any collection of $k\geq 2$ disjoint subsets $J_1, \ldots, J_k \in N$, their union $J_1\cup \cdots \cup J_k$ is not in $\fB$.\label{n2}
    \end{enumerate}
    Let $V_\fB=\fB \setminus \{[n+1]\}$.
    A simplicial complex $K_\fB$ over $V_\fB$ called
    the \emph{nested set complex}  is defined to be the set of all nested sets for $\fB$.
\end{dfn}

It should be noted that $K_\fB$ is realizable as the boundary complex of a Delzant polytope $P_\fB$. Refer to \cite[Section~3]{Choi-Park2015} for details.
The normal vectors of facets of $P_\fB$ define a characteristic matrix $\Lambda_\fB$ over $K_\fB$: the column of $\Lambda_\fB$ indexed by $I\in \fB \setminus \{[n+1]\}$ is $\sum_{k \in I} \e_k$, where $\e_{n+1}=\sum_{i=1}^n \e_i$.
Then we obtain the real toric variety $X^\R_\fB:= M^\R(K_\fB,{\Lambda_\fB})$ associated to a connected building set $\fB$.
Notice that $V_\fB$ is a subset of $V_{A_n}$, and the corresponding column of $\Lambda_\fB$ is equal to that of $\Lambda_{A_n}$.

We consider the automorphism group $\Aut(\fB)$ of $\fB$ consisting of all elements of $\Sym_{n+1}$ which preserve $\fB$. Then, $\Aut(\fB) \subset \Sym_{n+1}$ acts on $\fB$ as well as $K_\fB$.

\begin{lem} \label{lem:building_set}
    Let $\fB$ be a connected building set.
    Then, $\Aut(\fB)$ acts on $X^\R_\fB$ and we have an $\Aut(\fB)$-module isomorphism
    \[
     H_\ast(X^\R_\fB) \cong \bigoplus_{S \in \row(\Lambda_\fB)}\widetilde{H}_{\ast-1}\left((K_\fB)_S \right).
    \]
\end{lem}
\begin{proof}
  For each $g \in \Aut(\fB) \subset \Sym_{n+1}$, there exists an $n \times n$-matrix $A_g$ such that $g \cdot \Lambda_{A_n} = A_g \Lambda_{A_n}$
  by Lemma~\ref{lem:Weyl_group_acts_well}. Since $\Lambda_\fB$ is obtainable from $\Lambda_{A_n}$ by removing columns, we have $g \cdot \Lambda_\fB = A_g \Lambda_{\fB}$.
   Therefore, the lemma follows from Theorem~\ref{stable-action}.
\end{proof}

Now, we consider a particular class of connected building sets on $[n+1]$ whose automorphism group is $\Sym_{n+1}$.
As in Section~\ref{subsec:Type_A}, we denote by $\cB_I$ the Boolean algebra over $I\subset [n+1]$,
 and by $\cB_I^Q$ the rank selected Boolean algebra over $I$ with respect to $Q \subset [n+1]$.

For $k \geq 1$, we set $Q_k = \{ 1, k+1, \ldots, n+1 \}$.
Then, $\fB_{n,k}=\cB^{Q_k}_{[n+1]}$ is a connected building set and $\Aut(\fB_{n,k}) = \Sym_{n+1}$.
We remark that this is a generalisation of $X^\R_{A_n}$ as $\fB_{n,1}$ for $k=1$ is the power set of $[n+1]$
and $X^\R_{\fB_{n,1}}$ is exactly equal to $X^\R_{A_n}$.

Similarly to the case $X^\R_{A_n}$, we identify an element $S \in \row(\Lambda_{\fB_{n,k}})$ with the even cardinality subsets $I_S$ of $[n+1]$,
then Lemma \ref{lem:building_set} specializes to giving an $\Sym_{n+1}$-module isomorphism $H_\ast(X^\R_{\fB_{n,k}}) \cong \bigoplus_{I_S\in {[n+1]\choose 2r}} \widetilde{H}_{\ast-1} (K_S)$.

\begin{lem} \label{lem:K_[n+1]}
  Let $n$ be an odd integer, and $S \in \row(\Lambda_{\fB_{n,k}})$ the sum of all row vectors of $\row(\Lambda_{\fB_{n,k}})$, that is, $I_S = [n+1]$.
  Then, $K_S$ is $\Sym_{n+1}$-equivariantly homeomorphic to $\cB^{Q}_{[n+1]}$, where
  $Q = \{1, 2, \ldots,k\} \cup \{1, 3, 5, \ldots, n \}$.
\end{lem}
\begin{proof}
    We observe that $V_{\fB_{n,k}}= \{J \subset [n+1] \mid |J|=1 \text{ or } k < |J| \le n \}$, and $K_S$ is the full-subcomplex of $K_{\fB_{n,k}}$ on the vertices $J \in V_{\fB_{n,k}}$ such that $|J|$ is odd.
    It should be noted that any set of $k$ singletons in $K_S$ form a ($k-1$)-simplex in $K_S$.
    Therefore, we can obtain $\cB^{Q}_{[n+1]}$ by taking the barycentric subdivision of such ($k-1$)-simplices, and hence, $K'_S$ is homeomorphic to $\cB^{Q}_{I_S}$.
    It is clear that it is compatible with the $\Sym_{n+1}$-action.
\end{proof}


By Theorem~\ref{thm:skew_Specht} and Remark~\ref{rmk:skew_Specht}, we obtain
\begin{thm}
  Let $n$ be an odd integer.
  The Foulkes representation appears on the top homology of $X^\R_{\fB_{n,k}}$.
  More precisely,
  we have an $\Sym_{n+1}$-module isomorphism
  \[
  H_\ast(X^\R_{\fB_{n,k}})\cong \begin{cases}
  0 & (*>d) \\
  S^{\lambda_Q} & (*=d),
  \end{cases}
  \]
 where  $d=\frac{n+1}{2} + \lfloor \frac{k}{2} \rfloor$,
$Q = \{1, 2, \ldots,k\} \cup \{1, 3, 5, \ldots, n \}$, $\lambda_Q$ is the skew hook with $n+1$ cells whose descent set is $Q$, and $S^{\lambda_Q}$ is the skew Specht module.
\end{thm}

\begin{proof}
    For $I_S = [n+1]$, by Lemma~\ref{lem:K_[n+1]}, $K_S$ is homeomorphic to $\cB^Q_{[n+1]}$ and
    $H_{d-1}(K_S)\cong S^Q$ by Theorem~\ref{thm:skew_Specht}.
    For $I_S \subsetneq [n+1]$, since $I_S$ has even cardinality, $|I_S| \leq n-1$.
    Recall that $K_S$ is the full-subcomplex of $K_{\fB_{n,k}}$ on the vertices $J \in V_{\fB_{n,k}}$ such that $|J \cap I_S|$ is odd.
    Since at most $k$ singletons can be adjacent in $K_S$, the dimension of $K_S$
     is less than $\frac{n-1}{2} + \lfloor \frac{k}{2} \rfloor = d-1$.
    \end{proof}

\begin{ex}
Fix $n=11$ and $k=4$ or $k=5$.
In this case, $d=8$, and $\lambda_{Q}$ is as follows:
$$
\tableau{
   & & & \hspace{.5cm} \\
   & & \hspace{.5cm}  & \hspace{.5cm}  \\
     &  \hspace{.5cm}&\hspace{.5cm}&  \\
       \hspace{.5cm}&\hspace{.5cm} & &\\
       \hspace{.5cm}&&&\\
       \hspace{.5cm}&&&\\
       \hspace{.5cm}&&&\\
       \hspace{.5cm}&&&\\
       \hspace{.5cm}&&&
   }$$
\end{ex}

\bibliographystyle{amsplain}

\end{document}